\documentclass[12pt]{article}

\usepackage{amssymb}
\usepackage{amsmath,amsthm}

\newtheorem{theorem}{Theorem}%[section]
\newtheorem{proposition}{Proposition}%[section]
\newtheorem{lemma}{Lemma}%[section]
\newtheorem{definition}{Definition}%[section]
\newtheorem{example}{Example}%[section]
\newcommand{\R}{\mathbb R}
\newcommand{\RR}{\overline{\mathbb R}}
\newcommand{\E}{\mathbb R^s}

\newcommand{\pr}{\prime}
\newcommand{\eps}{\varepsilon}

\newcommand{\scalpr}[2]{\langle#1,#2\rangle}

\begin{document}

\title{Optimality Conditions in Vector Optimization}

\author{Vsevolod I. Ivanov \thanks{Department of Mathematics, Technical University of Varna, Bulgaria.
E-mail: vsevolod.ivanov@tu-varna.bg}
}

\maketitle

\begin{abstract}
In this paper we obtain second- and first-order optimality conditions of Kuhn-Tucker type and Fritz John one for weak efficiency
in the vector problem with inequality constraints. In the necessary conditions we suppose that the objective function and the
active constraints are continuously differentiable. We introduce notions of KTSP-invex problem  and second-order KTSP-invex one. We obtain that the vector problem is (second-order) KTSP-invex if and only if for every triple $(\bar x,\bar\lambda,\bar\mu)$ with Lagrange multipliers $\bar\lambda$ and $\bar\mu$ for the objective function and constraints respectively, which satisfies the (second-order) necessary optimality conditions, the pair $(\bar x,\bar\mu)$ is a saddle point of the scalar Lagrange function with a fixed multiplier $\bar\lambda$. We introduce notions second-order KT-pseudoinvex-I,  second-order KT-pseudo\-invex-II,  second-order KT-invex problems. We prove that every second-order Kuhn-Tucker stationary point is a weak  global  Pareto minimizer  (global  Pareto minimizer) if and only if the problem is second-order KT-pseudoinvex-I (KT-pseudoinvex-II). 
It is derived that every second-order Kuhn Tucker stationary point is a global solution of the weighting problem if and only if the vector problem is second-order KT-invex. 
\end{abstract}



\section{Introduction}
\setcounter{equation}{0}
\label{s1}

Kuhn-Tucker  optimality conditions and Fritz John ones  for scalar and vector nonsmooth problems with inequality constraints are among the most important directions of investigation in optimization. %We mention several works where second-order conditions have been studied: \cite{ctr03,gz96,gfr06,ggr05,JMAA2008,gl03,gut10,husn84,hua94,JMAA2009,OptLett-1,jw99,luc02,mae04,rtr89,stu91}

In the present paper we deal with optimality conditions of Kuhn-Tucker type for the following problem:
\medskip
%\noindent

Minimize$\quad f(x)\quad$subject to$\quad x\in X,\;
g(x)\leqq0$,\hfill  (P)
\medskip

\noindent
where $f:X\to\R^n$ and $g:X\to\R^m$ are given vector real-valued functions, defined on an open set $X\subset\R^s$. Denote by $S$ the set of feasible points, that is $S:=\{x\in X\mid g(x)\leqq 0\}$. All results given here are obtained for nonsmooth problems in terms of the usual second-order directional derivative. 

For every vector $\lambda$ from the set 
\[
\Lambda:=\{\lambda=(\lambda_1,\lambda_2,\ldots,\lambda_n)\in\R^n\mid\sum_{i=1}^n\lambda_i=1,\;\lambda_i\geqq 0,\; i=1,2,\ldots, n\}
\]
we consider the following weighting scalar problem:
\medskip

%\noindent
Minimize$\quad\scalpr{\lambda}{f(x)}\quad$subject to$\quad x\in X,\;
g(x)\leqq0$.\hfill ${\rm (P_\lambda)}$
\medskip

\noindent
Here we denote by $\scalpr{a}{b}$ the scalar product of the vectors $a$ and $b$.
For every $\lambda\in\Lambda$ we define the scalar Lagrange function
\[
L_\lambda(x,\mu)=\scalpr{\lambda}{f(x)}+\scalpr{\mu}{g(x)}.
\]
For all vectors $\lambda\in\Lambda$ and $\mu\geqq 0$ we consider the following unconstrained problem
\medskip

%\noindent
Minimize$\quad\scalpr{\lambda}{f(x)}+\scalpr{\mu}{g(x)}\quad$subject to$\quad x\in X$. \hfill ${\rm (P_{\lambda,\mu})}$
\medskip

These problems are formulated and discussed by Geoffrion \cite{geo68}.

It is easy to prove the following claims:

\begin{proposition}
Let $\bar x\in S$ be a global solution of ${\rm (P_\lambda)}$ for some $\lambda\in\Lambda$. Then $\bar x$ is a weakly efficient solution of {\rm (P)}.
\end{proposition}

\begin{proposition}
Let $\bar x\in S$ be a global minimizer of ${\rm (P_{\lambda,\mu})}$ for some $\lambda\in\Lambda$ and $\mu\geqq 0$ with $\mu_j g_j(\bar x)=0$ for all $j=1,2,\dots,m$. Then $\bar x$ is a global solution solution of ${\rm (P_\lambda)}$.
\end{proposition}

\begin{proposition}\label{pr3}
Let $\bar x\in S$ be a global minimizer of ${\rm (P_{\lambda,\bar\mu})}$ for some $\lambda\in\Lambda$ and $\bar\mu\geqq 0$ with $\bar\mu_j g_j(\bar x)=0$. Then the pair $(\bar x,\bar\mu)$ is a Kuhn-Tucker saddle point of ${\rm L_\lambda}$.
\end{proposition}

\begin{proposition}\label{pr4}
Let the pair $(\bar x,\bar\mu)\in S\times [0,+\infty)^m$ be a Kuhn-Tucker saddle point of ${\rm L_\lambda}$ for some $\lambda\in\Lambda$. Then
$\bar x\in S$ is a global minimizer of ${\rm (P_{\lambda,\bar\mu})}$  and $\bar\mu_j g_j(\bar x)=0$. 
\end{proposition}
Denote the solution sets of (P), (P$_\lambda$) and (P$_{\lambda,\mu}$) respectively by WE$\,$(P), Argmin$\,$(P$_\lambda$) and Argmin$\,$(P$_{\lambda,\mu}$), and the set of points which satisfy Kuhn-Tucker necessary optimality conditions by KT$\,$(P). We have the following implications under additional assumptions:
\[
\bar x\in {\rm Argmin\,(P_{\lambda,\mu})}\;\Rightarrow\; \bar x\in {\rm  Argmin\,(P_\lambda)}\;\Rightarrow\;\bar x\in {\rm WE\,(P)} \;\Rightarrow\; \bar x\in {\rm KT\,(P)}.
\]
It is interesting when the following converse conditions are satisfied:

1) Every Kuhn-Tucker stationary point is a saddle point of the Lagrange function.

2) Every Kuhn-Tucker stationary point is a  weak global Pareto minimizer.

3) Every Kuhn-Tucker stationary point is a solution of the weighting problem. 

We answer the first question in the first- and second-order cases. 
The answers of the second question and the third one are known. \cite{ara08,osu99,osu98} In this paper we find the largest class of functions such that every second-order Kuhn-Tucker stationary point is weakly efficient, and when it is a solution of weighting problem. 

We define the following new classes of vector problems: KTSP-invex, second-order KTSP-invex, second-order KT-pseudoinvex I, second-order KT-pseudoinvex II, second-order KT-invex. We prove that in contrast of the second-order generalized invex vector problems, applied in Duality Theory, the classes of second-order generalized invex vector problems, introduced in the paper, include the respective class of first-order generalized invex vector problem.

In Section \ref{s2} we prove a theorem of the alternative, which we apply later several times. In Section \ref{s3}
we derive second-order necessary conditions for a local minimum of problems with C$^1$ (i.e. continuously differentiable) data. 

The necessary conditions are not enough to determine the minimizer. In the next sections we show that the same conditions without constraint qualification are sufficient for a global minimum in some problems of generalized convexity type. In Section \ref{s6} we introduce notions of KTSP-invex and second-order KTSP-invex problems with inequality constraints. We prove that these classes are the largest ones with the following properties: every (second-order) Kuhn-Tucker stationary point is saddle for the respective scalar Lagrange function. In other words we find the largest classes of problems (P) such that the saddle points of the Lagrange function and the (second-order) Kuhn-Tucker stationary points coincide with the weakly efficient solution of (P). 
In Section \ref{s4} we define a notion of a second-order KT-pseudoinvex-I vector problem with inequality constraints. We obtain that these problems are the largest class such that every second-order Kuhn-Tucker stationary point is a weak global Pareto minimizer of (P). At last, in Section \ref{s5} we define a notion of a second-order invex vector problem with inequality constraints. We prove that these problems are the largest class such that every second-order Kuhn-Tucker stationary point is a global solution of the weighting problem. In other words we find the largest class of vector problems with inequality constraints  such that the set of weakly efficient global solutions of (P) coincides with the global solutions of the weighting problem and the second-order Kuhn-Tucker stationary points. 

The notions of invexity that we study in this paper are generalizations of some notions for scalar and vector problems
\cite{ara08,osu99,osu98,mar85,JOGO,Optimization,OptLett-2}.

\section{A Theorem of the Alternative}
\label{s2}\setcounter{equation}{0}
Throughout this paper we use the following notations comparing the vectors $x$ and $y$ with components $x_i$ and $y_i$ in finite-dimensional spaces:
\[
x>y\quad\textrm{iff}\quad x_i>y_i\quad\textrm{for all indeces }i;
\]
\[
x\geqq y\quad\textrm{iff}\quad x_i\geqq y_i\quad\textrm{for all indeces }i;
\]
\[
x\ge y\quad\textrm{iff}\quad x_i\geqq y_i\quad\textrm{for all indeces }i,\; x\ne y.
\]
We denote by $e$ the vector $e=(1,1,\ldots,1)$, and by $A^T$ the transpose of the matrix $A$.

We apply several times the following lemma:

\begin{lemma}\label{th-alternative}
Let $A$, $B$, $C$, and $D$ be given matrices. Then either the system
\begin{equation}\label{7}
A^Tx+C^Tu<0,\quad B^Tx+D^Tu\leqq 0,\quad u\geqq 0
\end{equation}
has a solution $(x,u)$ or the system
\begin{equation}\label{8}
y^TA+z^TB=0,\quad y^TC+z^TD\geqq 0,\quad y\ge 0,\quad z\geqq 0
\end{equation}
has a solution $(y,z)$, but never both.
\end{lemma}  
\begin{proof}
Consider the following primal linear programming problem

\medskip
\begin{tabular}{ll}
Minimize & 0         \\
subject to & $y^TA+z^TB=0$  \\
& $y^TC+z^TD\geqq 0$ \\
& $\scalpr{y}{e}=1$ \\
& $y\geqq 0,\quad z\geqq 0$
\end{tabular}
\medskip

\noindent and its dual one

\medskip
\begin{tabular}{lll}
Maximize & $v$ & \\
subject to & $A^Tx+C^Tu+ve$ & $\leqq 0$ \\
& $B^Tx+D^Tu$ & $\leqq 0$  \\
&  $u\geqq 0$. &
\end{tabular}
\medskip

Suppose that the system (\ref{7}) has no solutions. We prove that the system (\ref{8}) is solvable. The dual problem is feasible, because $x=0$, $u=0$, $v=0$ is a feasible point. It follows from the assumption (\ref{7}) is unsolvable that there is no a number $v>0$ and vectors $x$ and $u$ such that the triple
$(x,u,v)$ is feasible for the dual problem. Therefore, the dual problem is solvable and its optimal value is 0. According to Duality Theorem in Linear Programming the primal problem is solvable with the same optimal value. Hence, the system (\ref{8}) is solvable.

Suppose that (\ref{7}) is solvable. We prove that (\ref{8}) has no solutions. Assume the contrary, that is (\ref{8}) is solvable. Therefore, the primal problem is solvable. It follows from Duality Theorem that the dual problem is also solvable and its optimal value is 0. Therefore, there is no a positive number $v$ and vectors $x$, $u$ such that the triple $(x,u,v)$ is feasible for the dual problem, which is a contradiction, because we suppose that (\ref{7}) is solvable.
\end{proof}

Motzkin's, Farkas', and Gordan's theorems of the alternative are particular cases of Lemma \ref{th-alternative} (see \cite{man69}).

\section{Second-Order Necessary Conditions for \\ Weak Local Minimum}
\label{s3}\setcounter{equation}{0}
In this section we derive necessary optimality conditions for the problem
{\rm{(P)}} with continuously differentiable data.

We begin this section with some preliminary definitions.
\smallskip
Denote by $\R$ the set of reals and $\RR=\R\cup\{-\infty\}\cup\{+\infty\}$.
Let the function $h:X\to\R$ with an open domain $X\subset\R^s$ be differentiable at
the point $x\in X$. Then the second-order directional derivative
$h^{\pr\pr}(x,u)$ of $h$ at the point $x\in X$ in direction $u\in\R^n$
is defined as an element of $\RR$ by
\begin{displaymath}
h^{\pr\pr}(x,u)=\lim_{t\to +0}\,2\,t^{-2}\,[h(x+tu)-h(x)-t\nabla h(x)u].
\end{displaymath}
The function $h$ is called second-order directionally differentiable at the point $x$
if the derivative $h^{\pr\pr}(x,u)$ exists for each direction $u\in\R^n$ and $-\infty<h^{\pr\pr}(x,u)<+\infty$. We say that $h^{\pr\pr}(x,u)$ exists, if it is finite.

Consider the problem {\rm{(P)}}. 
For every feasible point $x\in S$  let $A(x)$ be the set of active constraints
\[
A(x):=\{j\in\{1,2,\ldots,m\}\mid g_j(x)=0\}.
\]

\begin{definition}
A direction $d$ is called critical at the point $x\in S$ if
\[
\nabla f_i(x)d\leqq0,\; i=1,2,\ldots,n,\,\quad \nabla g_j(x)d\leqq 0,\;\forall j\in A(x).
\]
For every critical direction $d$ at the feasible point $x$ denote by $I(x,d)$ and $J(x,d)$ the following index sets:
\[
I(x,d):=\{i\in\{1,2,\ldots,n\}\mid\nabla f_i(x)d=0\},
\]
\[
J(x,d):=\{j\in A(x)\mid\nabla g_j(x)d=0\}.
\]
\end{definition}

\begin{definition}
A feasible point $\bar x\in S$ is called a (weak) local Pareto minimizer, or (weakly) efficient if there exists a neigbourhood $U\ni\bar x$ such that there is no $x\in U\cap S$ with  $f(x)\le f(\bar x)$ ($f(x)<f(\bar x)$). 
The point $\bar x\in S$ is called a (weak) global Pareto minimizer if there does not exist $x\in S$ with $f(x)\le f(\bar x)$ ($f(x)<f(\bar x)$).
\end{definition}
The problem (P) is said to be Fr\'echet differentiable if $f$ and $g$ are Fr\'echet differentiable.

The following result, which contains first-order necessary optimality conditions under various constraint qualifications,  is known as Kuhn-Tucker Theorem:
\begin{proposition}
Let the problem {\rm (P)} be Fr\'echet differentiable on some open set $X\subseteq\R^s$. Suppose that $\bar x$ is a weak local Pareto minimizer, and any of the constraint qualifications holds. Then there exist vectors  $\lambda=(\lambda_1,\lambda_2,\ldots,\lambda_n)\ge 0$ and 
$\mu=(\mu_1,\mu_2,\ldots,\mu_m)\geqq 0$ such that
\[
\sum_{i=1}^n\lambda_i\nabla f_i(\bar x)+\sum_{j=1}^m\mu_j\nabla g_j(\bar x)=0,\;\lambda_j\, g_j(\bar x)=0,\; j=1,2,...,m.
%\end{equation}
\]
\end{proposition}

The next two theorems are generalizations of the primal and dual necessary optimality conditions by Ginchev, Ivanov \cite{JMAA2008} where scalar problems with ineuality constraints are treated. 

\begin{theorem}\label{th1}
Let $X$ be an open set in the space $\R^s$, the functions
\[
f_i\; (i=1,2,\ldots,m),\quad g_j\; (j=1,2,\ldots,m)
\]
be defined on $X$. Suppose that $\bar x$ is a weak local Pareto minimizer of the problem {\rm{(P)}},
the functions $g_j$ $(j\notin A(\bar x))$ are continuous at $\bar x$,
the functions $f$ and $g_j$ $(j\in\ A(\bar x))$ are continuously differentiable,
and the functions $f_i$ $(i\in I(\bar x,d))$ and $g_j$ $(j\in J(\bar x,d))$ are second-order directionally
differentiable at $\bar x$ in any critical direction $d\in\R^s$. Then for every critical direction $d\in\E$, 
it follows that there is no $z\in\E$ which solves the system
\begin{equation}\label{1}
\nabla f_i(\bar x)z+f_i^{\pr\pr}(\bar x,d)<0,\quad i\in I(\bar x,d).
\end{equation}
\begin{equation}\label{2}
\nabla g_j(\bar x)z+g_j^{\pr\pr}(\bar x,d)<0,\quad j\in J(\bar x,d).
\end{equation}
\end{theorem}

\begin{proof}
Suppose the contrary that there exists a critical direction $d$ such that
the system (\ref{1}), (\ref{2}) has a solution $z\in\E$. Obviously, the case
$I(\bar x,d)\cup J(\bar x,d)=\emptyset$ is impossible, because $\bar x$ is a weak minimimer.

Let $i\in \{1,2,\ldots,n\}$ be arbitrary fixed. Consider the function of one variable 
\[
\varphi_i(t)=f_i(\bar x+td+0.5t^2z).
\]
Since $X$ is open and $\bar x$ is feasible, there exists $\eps_i>0$ such that $\varphi_i$ is defined for $0\leqq t<\eps_i$. We have
\[
\varphi_i^\pr(t)=\nabla f_i(\bar x+td+0.5t^2z)(d+tz).
\]
Therefore, $\varphi_i^\pr(0)=\nabla f_i(\bar x)d$. Consider the
differential quotient
\[
2\,t^{-2}\,[\varphi_i(t)-\varphi_i(0)-t\varphi^\pr_i(0)]
=2\,t^{-2}\,[f_i(\bar x+td+0.5t^2 z)-f_i(\bar x)-t\nabla f_i(\bar x)d].
\]
According to the Mean-Value Theorem there exists $\theta_i\in(0,1)$ such that
\[
f_i(\bar x+td+0.5t^2 z)=f_i(\bar x+td)+\nabla
f_i(\bar x+td+0.5t^2\theta_i z)(0.5t^2 z).
\]
By $f_i\in\rm{C}^1$, we obtain that there exists the second-order
directional derivative $\varphi_i^{\pr\pr}(0,1)$ and
\[
\nabla f_i(\bar x)z+f_i^{\pr\pr}(\bar x,d)=
\lim_{t\to+0}\nabla f_i(\bar x+td+0.5t^2\theta_i z)z
\]
\[
+\lim_{t\to+0}\;2t^{-2}(f_i(\bar x+td)-f_i(\bar x)-
t\nabla f_i(\bar x)d)=\varphi_i^{\pr\pr}(0,1).
\]
Since $z$ is a solution of the system (\ref{1}) with a direction $d$,
we conclude that for every $i\in \{1,2,\ldots,n\}$
there exists $\varepsilon_i>0$ such that
\[
\varphi_i(t)-\varphi_i(0)-t\varphi^\pr_i(0))<0\quad {\rm for all}\quad t\in(0,\varepsilon_i)
\]
that is
\begin{equation}\label{3}
f_i(\bar x+t d+0.5t^2 z)-f_i(\bar x)-t\nabla f_i(\bar x)d<0
%-t\nabla f_i(\bar x)d<0
\quad\textrm{ for all }\quad t\in(0,\varepsilon_i).
\end{equation}

Consider the function of one variable 
\[
\psi_j(t)=g_j(\bar x+td+0.5t^2z).
\]
Using similar arguments we prove that there exists $\delta_j>0$, $j=1,2,\ldots, m$ with
\begin{equation}\label{5}
g_j(\bar x+t d+0.5t^2 z)-g_j(\bar x)-t\nabla g_j(\bar x)d<0
\quad\textrm{ for all }\quad t\in(0,\delta_j).
\end{equation}

Consider the following cases:

1) For every $j\in\{1,2,...,m\}\setminus A(\bar x)$ we have $g_j(\bar x)<0$.
Hence, by continuity, there exists $\delta_j>0$ such that
$g_j(\bar x+t d+0.5t^2 z)<0$ for all $t\in[0,\delta_j)$.

2) For every $j\in A(\bar x)\setminus J(\bar x,d)$ we have
$\nabla g_j(\bar x)d=
\psi_j^\prime(0)<0$.
Therefore, there exists $\delta_j>0$ such that
$\psi_j(t)<\psi_j(0)$ for all $t\in(0,\delta_j)$. Hence we have
$g_j(\bar x+t d+0.5t^2 z)<g_j(\bar x)=0$ for all $t\in(0,\delta_j)$.

3) For all $j\in J(\bar x,d)$, by $\nabla g_j(\bar x)d=0$,
it follows from (\ref{5}) that there exist $\delta_j>0$ such that
$g_j(\bar x +t d+0.5t^2 z)<g_j(\bar x)=0$ for all $t\in(0,\delta_j)$.

4) For all $i\notin I(\bar x,d)$ we have $\nabla f_i(\bar x)d<0$ that is
$\varphi_i^\pr(0)<0$ and therefore, for some $\varepsilon_i>0$ it holds
$f_i(\bar x+t d+0.5t^2 z)<f_i(\bar x)$ for all $t\in(0,\varepsilon_i)$.

5) For all $i\in I(\bar x,d)$ we have  $\nabla f_i(\bar x)d=0$ and according to (\ref{3})
there exists $\varepsilon_i>0$ such that
$f_i(\bar x +t d+0.5t^2 z)<f_i(\bar x)$ for all $t\in(0,\varepsilon_i)$.

It is seen that $\bar x$ is not a weak local minimizer, contradicting our
hypothesis.
\end{proof}

\begin{theorem}\label{NKT}
If all hypotheses of Theorem \ref{th1} are satisfied, then corresponding
to any critical direction $d$
there exist nonnegative multipliers 
\[
\lambda=(\lambda_1,\lambda_2,\ldots,\lambda_n)\quad\textrm{and}\quad \mu=(\mu_1,\mu_2,\ldots,\mu_m)\quad\textrm{with}\quad (\lambda,\mu)\ne 0
\]
such that
\begin{equation}\label{6}
\mu_j g_j(\bar x)=0,\;j=1,2,\ldots,m,\quad\nabla L(\bar x)=0,
\end{equation}
\begin{equation}\label{4}
\lambda_i\nabla f_i(\bar x)d=0,\quad i=1,2,\ldots,n,\quad\mu_j\nabla g_j(\bar x)d=0,\quad j=1,2,\ldots,m,
\end{equation}
\begin{equation}\label{10}
L^{\pr\pr}(\bar x,d)=\sum_{i=1}^n\lambda_i
f_i^{\pr\pr}(\bar x,d)+ \sum_{j\in A(\bar x)}\mu_j g_j^{\pr\pr}(\bar x,d)\geqq 0,
\end{equation}
where $L=\scalpr{\lambda}{f}+\scalpr{\mu}{g}$ is the Lagrange function.
Suppose further that a constraint qualification holds. Then we have $\lambda\ne 0$.
\end{theorem}

\begin{proof}
Consider the matrix $A$ whose columns are 
\[
\{\nabla f_i(\bar x)\mid i\in I(\bar x,d)\}\quad\textrm{and}\quad  \{\nabla g_j(\bar x)\mid j\in J(\bar x,d)\},
\]
and the row-matrix $C$  whose components are
\[
\{f_i^{\pr\pr}(\bar x,d)\mid i\in I(\bar x,d)\}\}\quad\textrm{and}\quad \{g_j^{\pr\pr}(\bar x,d)\mid j\in J(\bar x,d)\}\}.
\]
With these notations Theorem \ref{th1} claims that the linear system $A^Tz+C^Tu<0$, $u\geqq 0$, where $u\in\R$, has no solutions. 
It follows from Lemma \ref{th-alternative} that there exist vectors $\lambda\in\R^n$ and $\mu\in\R^m$ with $\lambda\geqq 0$, $\mu\geqq 0$, $(\lambda,\mu)\ne (0,0)$ such that conditions (\ref{6}) and (\ref{10}) are satisfied. According to the assumptions $d$ is critical and $\nabla L(\bar x)=0$ we obtain that conditions (\ref{4}) are satisfied.
\end{proof}

\section{The Saddle Points of the Lagrange Function and Weak Efficiency}
\label{s6}\setcounter{equation}{0}

Consider the vector problem (P) and the scalar Lagrange function
\[
L_\lambda(x,\mu)=\scalpr{\lambda}{f(x)}+\scalpr{\mu}{g(x)},\quad\lambda\in\Lambda,\;\mu\geqq 0.
\]

\begin{definition}
A point $(\bar x,\bar\mu)\in X\times [0,+\infty)^m$ is called a Kuhn-Tucker saddle point of the Lagrange function if 
\begin{equation}\label{9}
L_{\lambda}(\bar x,\mu)\leqq L_{\lambda}(\bar x,\bar\mu)\leqq L_{\lambda}(x,\bar\mu),\quad\forall x\in X,\;\forall\mu\geqq 0.
\end{equation}
\end{definition}

Denote the vector $(\nabla f_1,\nabla f_2,\ldots,\nabla f_n)$ by $\nabla f$. 

The following definition is an extension to vector problems of the respective notion by Ivanov \cite{OptLett-2} which holds for scalar minimization problems.

\begin{definition}
Let {\rm (P)} be a Fr\'echet differentiable problem. We call {\rm (P)} Kuhn-Tucker saddle point invex (for short, KTSP-invex) if
there exists a vector function $\eta: X\times X\to\R^s$  such that the following implication holds:

\medskip
\noindent
$\left.
\begin{array}{l}
x\in X,\; \\
u\in X,\; \\
g(u)\leqq 0 
\end{array}
\right]
\quad$
imply
$\quad
\left[
\begin{array}{l}
\textrm{either } f(x)-f(u)\geqq\nabla f(u)\eta(x,u), \\
g_{j}(x)\geqq\nabla g_{j}(u)\eta(x,u),\; j\in A(u) \\
\textrm{or } \nabla f(u)\eta(x,u)<0, \\
\nabla g(u)\eta(x,u)\leqq 0.
\end{array}
\right.$
\hfill {\rm (KTSPI)}
\end{definition}
\medskip

\begin{theorem}\label{th3}
Let the  problem {\rm (P)} be Fr\'echet differentiable.  Then {\rm (P)} is KTSP-invex if and only if,
for every triple $(\bar x,\bar\lambda,\bar\mu)$  which satisfies Kuhn-Tucker optimality conditions
\begin{equation}\label{12}
\scalpr{\bar\lambda}{\nabla f(\bar x)}+\sum_{j\in A(\bar x)}\bar\mu_j\nabla g_j(\bar x)=0,\quad\bar x\in S,\; \bar\lambda\ge 0,\; \bar\mu\geqq 0,
\end{equation}
the pair $(\bar x,\bar\mu)$ is a saddle point of the Lagrange function $L_{\bar\lambda}$.
\end{theorem}
\begin{proof}
Suppose that (P) is KTSP-invex. Let the triple $(\bar x,\bar\lambda,\bar\mu)$ satisfy Kuhn-Tucker conditions. We prove that $(\bar x,\bar\mu)$ is a saddle
point of Lagrange function. Suppose that $x$ is an arbitrary point from $X$. The left inequality in (\ref{9}) follows from $\bar x\in S$, $\bar\mu\geqq 0$,  $\mu\geqq 0$, and $\bar\mu_j g_j(\bar x)=0$, $j=1,2,...,m$. Therefore, it is enough to prove that
\begin{equation}\label{11}
\scalpr{\bar\lambda}{f(x)-f(\bar x)}+\scalpr{\bar\mu}{g(x)}\geqq 0.
\end{equation}

It follows from Lemma \ref{th-alternative}, by (\ref{12}), that the system
\begin{equation}\label{18}
\nabla f(\bar x)\eta<0,\quad \nabla g(\bar x)\eta\leqq 0
\end{equation}
is inconsistent. According to KTSP-invexity there exists $\eta\in\R^s$ such that
\begin{equation}\label{13}
f_i(x)-f_i(\bar x)\geqq \nabla f_i(\bar x)\eta,\quad i=1,2,\ldots, n
\end{equation}
\begin{equation}\label{14}
g_j(x)\geqq \nabla g_j(\bar x)\eta,\quad j\in A(\bar x). 
\end{equation}
Let us multiply (\ref{13}) by $\bar\lambda_i$ and (\ref{14}) by $\bar\mu_i$ and add all obtained inequalities. Then, we can see that (\ref{11}) is satisfied by Kuhn-Tucker conditions (\ref{12}), which implies that $(\bar x,\bar\mu)$ is a saddle point of $L_{\bar\lambda}$.

Suppose that every pair $(\bar x,\bar\mu)$, which satisfies conditions (\ref{12}) together with the Lagrange multiplier $\bar\lambda$,  is a saddle point of the Lagrange function $L_{\bar\lambda}$. We prove that (P) is KTSP-invex. Choose an arbitrary point $x\in X$.    Consider the linear programming problem with variables $\lambda_1, \lambda_2,\ldots,\lambda_n$, $\mu_j$, $j\in A(\bar x)$

\medskip
\noindent
\begin{tabular}{@{}ll}
Minimize & $\sum_{i=1}^n\lambda_i\left[f_i(x)-f_i(\bar x)\right]+$ $\sum_{j\in A(\bar x)} \mu_j\,g_j(x)$ \\
\smallskip\noindent
subject to & $\sum_{i=1}^n\lambda_i \nabla f_i(\bar x)+\sum_{j\in A(\bar x)} \mu_j\nabla g_j(\bar x)=0$, \\
& $ \sum_{i=1}^n\lambda_i=1$, \\
&  $\lambda_i\geqq 0$, $i=1,2,...,m$,\quad $\mu_j\geqq 0$, $j\in A(\bar x)$.
\end{tabular}

\medskip
It follows from the hypothesis that the objective function of this problem is non-negative over the feasible set or  the feasible set is empty. Therefore this problem is solvable with a non-negative minimal value or it is unsolvable when its feasible set is empty. The dual problem is the following one:

\medskip
\noindent
\begin{tabular}{@{}lll}
Maximize & $v$ & \\
subject to & $\nabla f(\bar x)\eta+v$ & $\leqq f(x)-f(\bar x)$ \\
&  $\nabla g_j(\bar x)\eta$  & $\leqq g_j(x)$, $j\in A(\bar x)$. \\
\end{tabular}
\medskip

Consider the case when the primal problem is solvable.
According to Duality Theorem in Linear Programming the dual problem is solvable, too, with the same optimal value. Therefore, there exist $v\ge 0$ and $\eta\in\R^s$, which satisfy the constraints of the dual problem.

Consider the case when the primal problem is infeasible. By Lemma \ref{th-alternative} the infeasibility of the primal problem implies that the system 
(\ref{18}) for $\eta$ is consistent. Both cases bring that (P) is KTSP-invex.
\end{proof}

We introduce the following definition:
\begin{definition}
Let {\rm (P)} be a Fr\'echet differentiable problem. We call {\rm (P)} second-order Kuhn-Tucker saddle point invex (for short, second-order KTSP-invex) if for every $x\in X$, $u\in S$, and for every critical direction $d$ at $u$ such that there exist the derivatives $f^{\pr\pr}(u,d)$, $g^{\pr\pr}(u,d)$ with finite values at least one of the systems for $\eta$ and $\omega$

\medskip
$\left\{
\begin{array}{l}
\nabla f_i(u)\eta+\omega f^{\pr\pr}_i(u,d)\leqq f_i(x)-f_i(u),\; i=1,2,\ldots,n \\
\nabla g_j(u)\eta+\omega g^{\pr\pr}_j(u,d) \leqq g_j(x),\; j\in A(u), \\
\omega\geqq 0
\end{array}\right.$
\medskip

\noindent
and

\medskip
$\left\{
\begin{array}{l}
 \nabla f(u)\eta+\omega f^{\pr\pr}(u,d)<0, \\
\nabla g(u)\eta +\omega g^{\pr\pr}(u,d)\leqq 0, \\
\omega\geqq 0
\end{array}\right.$
\medskip

\noindent
has a solution.
\end{definition}

\begin{theorem}\label{th7}
Every Fr\'echet differentiable KTSP-invex problem with inequality constraints is second-order KTSP-invex.
\end{theorem}
\begin{proof}
Let (P) be KTSP-invex. Therefore, for all points $x\in X$ and $u\in X$ with $g(u)\leqq 0$ there exists $v\in\R^s$ such that either

\[
\nabla f(u)v\leqq f(x)-f(u),\quad\nabla g_{j}(u)v\leqq g_{j}(x),\; j\in A(u),
\] 
or
\[
\nabla f(u)v<0,\quad\nabla g(u)v\leqq 0.
\]
The choice $\eta=v$, $\omega=0$, and $d=0$ shows that (P) is second-order KTSP-invex,  because $f^{\pr\pr}(u,0)=0$ and $g_j^{\pr\pr}(u,0)=0$.
\end{proof}

\begin{theorem}\label{th4}
Let the  problem {\rm (P)} be Fr\'echet differentiable.  Then {\rm (P)} is second-order KTSP-invex if and only if
for every triple $(\bar x,\bar\lambda,\bar\mu)$,  which satisfies the second-order Kuhn-Tucker necessary optimality conditions {\rm (\ref{6})} and 
\rm{(\ref{10})}, the pair $(\bar x,\bar\mu)$ is a saddle point of the Lagrange function $L_{\bar\lambda}$.
\end{theorem}
\begin{proof}
Suppose that (P) is second-order KTSP-invex. Let the triple $(\bar x,\bar\lambda,\bar\mu)$ satisfy the Kuhn-Tucker conditions with respect to an arbitrary direction $d$ which is critical at $\bar x$. We prove that $(\bar x,\bar\mu)$ is a saddle
point of the Lagrange function. Suppose that $x$ is an arbitrary point from $X$. The left inequality in (\ref{9}) follows from $\bar x\in S$, $\bar\mu\geqq 0$, $\mu\geqq 0$, and  $\bar\mu_j g_j(\bar x)=0$, $j=1,2,...,m$. Therefore, it is enough to prove that inequality (\ref{11}) holds

It follows from Lemma \ref{th-alternative}, by (\ref{6}) and (\ref{10}) that the system
\[
\nabla f(\bar x)\eta+\omega f^{\pr\pr}(\bar x,d)<0,\quad \nabla g(\bar x)\eta +\omega g^{\pr\pr}(\bar x,d)\leqq 0,\quad\omega\ge0
\]
is inconsistent. According to KTSP-invexity there exists $\eta\in\R^s$ and $\omega\ge 0$ such that
\begin{equation}\label{16}
f_i(x)-f_i(\bar x)\geqq \nabla f_i(\bar x)\eta+\omega f^{\pr\pr}_i(\bar x,d),\quad i=1,2,\ldots, n
\end{equation}
\begin{equation}\label{17}
g_j(x)\geqq \nabla g_j(\bar x)\eta +\omega g^{\pr\pr}_j(\bar x,d),\quad j\in A(\bar x). 
\end{equation}
Let us multiply (\ref{16}) by $\bar\lambda_i$ and (\ref{17}) by $\bar\mu_i$ and add all obtained inequalities. Then, we can see that (\ref{11}) is satisfied by Kuhn-Tucker conditions (\ref{6}) and (\ref{10}), which implies that $(\bar x,\bar\mu)$ is a saddle point of $L_{\bar\lambda}$.

Suppose that $d$ is an arbitrary direction, and every pair $(\bar x,\bar\mu)$, which satisfies conditions (\ref{6}) and (\ref{10}) together with the Lagrange multiplier $\bar\lambda$,  is a saddle point of the Lagrange function $L_{\bar\lambda}$. We prove that (P) is second-order KTSP-invex. Choose an arbitrary point $x\in X$.    Consider the linear programming problem with variables $\lambda_1, \lambda_2,\ldots,\lambda_n$, $\mu_j$, $j\in A(\bar x)$

\medskip
\noindent
\begin{tabular}{@{}ll}
Minimize & $\sum_{i=1}^n\lambda_i\left(f_i(x)-f_i(\bar x)\right)+$ $\sum_{j\in A(\bar x)} \mu_j\,g_j(x)$ \\
%\smallskip\noindent
subject to & $\sum_{i=1}^n\lambda_i \nabla f_i(\bar x)+\sum_{j\in A(\bar x)} \mu_j\nabla g_j(\bar x)=0$, \\
& $\sum_{i=1}^n\lambda_i f_i^{\pr\pr}(\bar x,d)+ \sum_{j\in A(\bar x)}\mu_j g_j^{\pr\pr}(\bar x,d)\geqq 0$, \\
& $ \sum_{i=1}^n\lambda_i=1$, \\
&  $\lambda_i\geqq 0$, $i=1,2,...,m$,\quad $\mu_j\geqq 0$, $j\in A(\bar x)$.
\end{tabular}

\medskip
It follows from the hypothesis that the objective function of this problem is non-negative over the feasible set or  the feasible set is empty. Therefore this problem is solvable with a non-negative minimal value or it is unsolvable when its feasible set is empty. The dual problem is the following one:

\medskip
\noindent
\begin{tabular}{@{}lll}
Maximize & $v$ & \\
subject to & $\nabla f(\bar x)\eta+\omega f^{\pr\pr}(\bar x,d)+v$ & $\leqq f(x)-f(\bar x)$, \\
&  $\nabla g_j(\bar x)\eta +\omega g^{\pr\pr}_j(\bar x,d)$ & $\leqq g_j(x)$, $j\in A(\bar x)$, \\
 & $ \omega\geqq 0$.
\end{tabular}
\medskip

Consider the case when the primal problem is solvable.
According to Duality Theorem in Linear Programming the dual problem is solvable, too, with the same optimal value. Therefore, there exist $v\ge 0$, $\eta\in\R^s$ and $\omega\in[0,+\infty)$, which satisfy the constraints of the dual problem.

Consider the case when the primal problem is infeasible. By Lemma \ref{th-alternative} the infeasibility of the primal problem implies that the system 
\[
\nabla f(\bar x)\eta +\omega f^{\pr\pr}(\bar x,d)<0,\quad
\nabla g(\bar x)\eta +\omega g^{\pr\pr}(\bar x,d)\leqq 0,\quad \omega\geqq 0.
\]
for $\eta$ and $\omega$ is consistent. Both cases bring that (P) is second-order KTSP-invex.
\end{proof}

The following example shows that the converse claim in Theorem \ref{th7} is not satisfied.

\begin{example}\label{ex3}
Consider the  problem (P) where $X\equiv\R^2$, $f=(f_1,f_2)$ is the vector function of two variables such that 
\[f_1(x_1,x_2)=(x_1^2+x_2^2)^2-2x_1^2+2x_2^2,\quad f_2(x_1,x_2)=(x_1^2-1)^2+2x_2^2.\] 
and $g$ is the scalar function 
\[
g(x_1,x_2)=x_1^2+x_2^2-1.
\]
The KT stationary points are the pairs $(0,0)$, $(1,0)$, $(-1,0)$ with Lagrange multipliers $\bar\mu=0$, $\lambda\in\Lambda$. The pair $(\bar x,\bar\mu)$ with $\bar x=(0,0)$ and $\bar\mu=0$ is not a  saddle point of $L_\lambda$ for every $\lambda=(\lambda_1,\lambda_2)\in\Lambda$, because 
\[
L_\lambda(x_{\eps},\bar\mu)=\lambda_1(\eps^4-2\eps^2)+\lambda_2(\eps^2-1)^2<L_\lambda(\bar x,\bar\mu),\; x_{\eps}=(\eps,0)
\]
 for all $\eps\in(0,\sqrt{2})$. It follows from Theorem \ref{th3} that this problem is not KTSP-invex. The point $(0,0)$ is not second-order stationary, because every direction is critical at $(0,0)$, and 
\[
\scalpr{d}{\nabla^2 L_\lambda(\bar x,\bar\mu)d}=-4d_1^2+4d_2^2<0,\;\bar x=(0,0),\;\bar \mu=0
\]
 if $d_2^2<d_1^2$, where $d=(d_1,d_2)$ is a direction. The points $(1,0)$ and $(-1,0)$ are second-order stationary, because the multipliers $\lambda_1=1$, $\lambda_2=0$, $\mu=0$, where $L=\lambda_1 f_1+\lambda_2 f_2+\mu g$ is the Lagrange function, satisfy the second-order necessary dual optimality conditions. Since these points compose with the multiplier $\mu=0$ saddle points of $L_\lambda$, then by Theorem \ref{th4}, the problem is second-order KTSP-invex.
\end{example}

\section{Optimality Conditions with Second-Order KT-Pseudoinvex Problems}
\label{s4}\setcounter{equation}{0}

Consider the problem (P). We denote  the vector 
\[
(f^{\pr\pr}_1(x,d), f^{\pr\pr}_2(x,d),\dots, f^{\pr\pr}_n(x,d))
\]
by $f^{\pr\pr}(x,d)$, and the set of all critical directions at $x$ by $D(x)$. 

The following notion  was introduced by Osuna-Gomez, Beato-Moreno, Rufian-Lizana \cite{osu99}:

%\begin{definition}[\cite{ara08,osu99}]\label{df2}
\begin{definition}\label{df2}
Let the  problem {\rm (P)} be Fr\'echet differentiable.
Then it is called KT-pseudoinvex-I if  there exists a map $\eta: X\times X\to\R^s$ 
such that the following implication holds:
\[
\left.
\begin{array}{l}
x\in X,\; y\in X,\;  f(y)<f(x) \\
g(x)\leqq 0,\; g(y)\leqq 0 \\
\end{array}\right]
\quad\Rightarrow \quad
\left[
\begin{array}{l}
\nabla f(x)\eta(x,y)<0 \\
\nabla g_{j}(x)\eta(x,y)\leqq 0, j\in A(x).
\end{array}
\right.
\]
\end{definition}

Definition \ref{df2} is a generalization of the notion of KT-invex problems due to Martin \cite{mar85} for scalar problems, because for such problems KT-invexity is equivalent to KT-pseudoinvexity. The following proposition is the main result in the paper of Osuma-Gomez, Rufian-Lizana, Ruiz-Canales \cite{osu98} (see Theorem 2.3).  It is a generalization of  Theorem 2.1  due to Martin \cite{mar85}.

\begin{proposition}\label{CharKT-I}
Let the problem (P) be Fr\'echet differentiable. Then it  is KT-pseudoinvex-I if and only if every KT stationary point is a weak  global minimizer.
\end{proposition}

Here we study the second-order case. Second-order KT-invex scalar problems are defined by Ivanov. \cite{JOGO}
We introduce the following notion, which is a second-order analog of Definition \ref{df2}.
\begin{definition}
Let {\rm (P)} be a Fr\'echet differentiable problem.
Then we call {\rm (P)} second-order KT-pseudoinvex-I if  there exist maps $d: X\times X\to\R^s$,
$\eta: X\times X\to\R^s$ and a function $\omega: X\times X\to [0,+\infty)$
such that the following implication holds:
\[
\left.
\begin{array}{l}
x\in X,\; y\in X \\
g(x)\leqq 0,\\
g(y)\leqq 0, \\
f(y)<f(x)
\end{array}\right]
\Rightarrow 
\left[
\begin{array}{l}
\exists f^{\pr\pr}(x,d(x,y)),\;\exists g^{\pr\pr}_j(x,d(x,y)),\; j\in A(x) \\
\nabla f(x)\eta(x,y)+
\omega(x,y) f^{\pr\pr}(x,d(x,y))<0 \\
\nabla g_{j}(x)\eta(x,y)+\omega(x,y) g^{\pr\pr}_{j}(x,d(x,y))\leqq 0,\; j\in A(x),\\
\textrm{where}\quad d(x,y)\in D(x).
\end{array}
\right.
\]
We suppose here that there exist $f^{\pr\pr}_i(x,d(x,y))$ and $g^{\pr\pr}_j(x,d(x,y))$ with finite values.
\end{definition}

\begin{theorem}\label{KT-I}
Every Fr\'echet differentiable KT-pseudoinvex-I problem with inequality constraints is second-order KT-pseudoinvex-I.
\end{theorem}
\begin{proof}
Let (P) be KT-pseudoinvex-I. Therefore, for all feasible points $x$ and $y$ with $f(y)<f(x)$ there exists $u\in\R^s$ such that
\[
\nabla f(x)u<0,\quad\nabla g_{j}(x)u\leqq 0,\; j\in A(x).
\]
The choice $\eta=u$, $\omega=0$, and $d=0$ shows that (P) is second-order KT-pseudoinvex-I,  because $f^{\pr\pr}(x,0)=0$ and $g_j^{\pr\pr}(x,0)=0$.
\end{proof}

\begin{definition}
A point $x\in S$ is called second-order Kuhn-Tucker stationary point (for short, second-order KT point) %if it is a KT stationary point and 
if for every critical direction $d\in\R^s$ at $x$ such that there exist the second-order derivatives $f^{\pr\pr} (x,d)$, $g_j^{\pr\pr} (x,d)$, $j\in A(x)$ with finite values, there are Lagrange multipliers $ \lambda= (\lambda_1, \lambda_2,\ldots, \lambda_n)\ge 0$, $\mu=(\mu_1,\mu_2,\ldots,\mu_m)\geqq 0$,  satisfying equations (\ref{6}) and (\ref{10}).
\end{definition}

Obviously, every second-order KT stationary point $x$ is a first-order KT stationary point. Indeed, for every Fr\'echet differentiable problem there is at least one critical direction $d\in\R^n$ such that there exist the second-order derivatives $f^{\pr\pr} (x,d)$, $g_j^{\pr\pr} (x,d)$, $j\in A(x)$. It is the direction $d=0$. Therefore, conditions (\ref{6}) and (\ref{10}) are satisfied when $d=0$. It follows from (\ref{6})  that $x$ is a KT stationary point. The converse implication does not hold.

\begin{theorem}\label{th2}
Let the problem {\rm (P)}  be Fr\'echet %and second-order directionally
differentiable. Then every second-order KT stationary point is a  weak global Pareto minimizer
if and only if {\rm (P)} is second-order KT-pseudoinvex-I.
\end{theorem}
\begin{proof}
Let (P) be second-order KT-pseudoinvex-I. We prove that every se\-cond-order KT point
is a weak global Pareto minimizer. Suppose the contrary that
$x$ is a second-order KT stationary point, but there exists $y\in S$ with
$f(y)<f(x)$. It follows from second-order pseudoinvexity that there exist a critical direction $d\in\R^s$, $\eta\in\R^s$, and a number $\omega\ge 0$ such that
\begin{equation}\label{21}
\nabla f_i(x)\eta(x,y)+
\omega(x,y) f^{\pr\pr}_i(x,d(x,y))<0,\; i=1,2,\ldots, n, 
\end{equation}
\begin{equation}\label{22}
\nabla g_j(x)\eta(x,y)+
\omega(x,y) g^{\pr\pr}_j(x,d(x,y))\leqq 0,\; \forall j\in A(x).
\end{equation}
We conclude from $x$ is a second-order KT point and $d$ is critical that there exist $\lambda\ge 0$ and $\mu\geqq 0$ with
\begin{equation}\label{23}
\sum_{i=1}^n\lambda_i\nabla f_i(x)+\sum_{j\in A(x)}\mu_j\nabla g_j(x)=0
\end{equation}
\begin{equation}\label{24}
\sum_{i=1}^n\lambda_i f^{\pr\pr}_i(x,d)+\sum_{j\in A(x)}\mu_j
g^{\pr\pr}_j(x,d)\geqq 0.
\end{equation}
Let us multiply (\ref{21}) by $\lambda_i$, (\ref{22}) by $\mu_j$ and add all obtained inequalities. Taking into account (\ref{23}), (\ref{24}) and 
$\lambda\ne 0$ we get the impossible inequality $0<0$.
 
Let each second-order KT point be a weak global Pareto minimizer. We prove that (P) is second-order KT-pseudoinvex-I. Take $x\in X$, $y\in X$ with $g(x)\leqq 0$, $g(y)\leqq 0$, $f(y)<f(x)$. Therefore, $x$ is not a weak global Pareto minimizer. Hence, $x$ is not a second-order KT point. Consequently, there is a critical direction $d$ such that the derivatives $f^{\pr\pr}(x,d)$, $g^{\pr\pr}(x,d)$ exist with finite values, 
but there are not  $(\lambda,\mu)\geqq 0$, $\lambda\ne 0$ which satisfy (\ref{23}) and (\ref{24}).
Therefore, the linear programming problem

\medskip
\begin{tabular}{ll}
Maximize & $\sum_{i=1}^n\lambda_i$ \\
subject to &
$\sum_{i=1}^n\lambda_i\nabla f_i(x)+\sum_{j\in A(x)}\mu_j\nabla g_j(x)=0$, \\
& $\sum_{i=1}^n\lambda_i f^{\pr\pr}_i(x,d)+\sum_{j\in A(x)}\mu_j g^{\pr\pr}_j(x,d)\ge 0,$ \\
& $\lambda_i\ge 0$,\quad $i=1,2,\ldots, n$,\quad  $\mu_j\ge 0,\quad\forall j\in A(x)$
\end{tabular}
\medskip

\noindent
has non-positive maximal value. If we take $\lambda_i=0$, $\mu_j=0$,
$j\in A(x)$, then we obtain that the exact optimal value of this
problem is 0. The dual linear problem is

\medskip
\begin{tabular}{ll}
Minimize & 0         \\
subject to & $\nabla f_i(x)u-v f^{\pr\pr}_i(x,d)\geqq 1$, $i=1,2,\ldots,n$  \\
& $\nabla g_j(x)u-v g^{\pr\pr}_j(x,d)\geqq 0,\; j\in A(x)$ \\
& $v\geqq 0$.
\end{tabular}
\medskip

\noindent
It follows from Duality Theorem that the dual problem is also solvable. Therefore,
there exist $\eta\in\R^s$ and $\omega\geqq 0$ ($\eta=-u$, $\omega=v$), which satisfy
inequalities (\ref{21}) and (\ref{22}). Hence, (P) is  second-order KT-pseudoinvex-I.
\end{proof}

The following examples show that the converse claim in Theorem \ref{KT-I} is not satisfied.

%\begin{example}\label{ex1}
%Consider the vector function of one variable $f=(f_1,f_2)$ such that $f_1=x^4-2x^2$, $f_2=-x^2$. The set of KT stationary points is 
%\[
%(-\infty,-1]\cup\{0\}\cup [1,+\infty).
%\]
% The point $x=0$ is not weakly efficient. Indeed, we have $f(\eps)<f(0)$ for $\eps\in(0,\sqrt{2})$. Therefore, the unconstrained problem is not KT-pseudoinvex-I. The set of second-order KT stationary points is $(-\infty,-1]\cup [1,+\infty)$ and each second-order KT stationary point is weakly efficient. Therefore, by Theorem \ref{th2}, the unconstrained problem is second-order KT-pseudoinvex-I.
%\end{example}

\begin{example}\label{ex2}
Consider the vector function of two variables $f=(f_1,f_2)$ such that 
\[f_1(x_1,x_2)=(x_1^2+x_2^2)^2-2x_1^2+2x_2^2,\quad f_2(x_1,x_2)=(x_1^2-1)^2+2x_2^2.\] 
The KT stationary points are the pairs $(0,0)$, $(1,0)$, $(-1,0)$. The vector function is not KT-pseudoinvex-I, because $(0,0)$ is not weakly efficient. For instance, we have $f(\eps,0)<f(0,0)$ for all $\eps\in(0,\sqrt{2})$. On the other hand $(0,0)$ is not a second-order stationary point, because
$\scalpr{d}{\nabla^2 L(0,0)\, d}<0$ for all critical directions $d=(d_1,d_2)$ such that $|d_1|>|d_2|$ where $L=\lambda_1 f_1+\lambda_2 f_2$ is the Lagrange function with a multiplier $\lambda=(\lambda_1,\lambda_2)$. The other stationary points $(1,0)$ and $(-1,0)$ are second-order stationary ones and they are global minimizers of the component $f_2$ and weakly efficient for $f$.    
Therefore, by Theorem \ref{th2} the unconstrained problem is second-order KT-pseudoinvex-I. 

It seem that the unconstrained minimization of $f$ could be reduced to the problem
\medskip

\begin{tabular}{ll}
%\begin{align}
Minimize  & $f(x,y)$   \\
subject to & $g_1(x,y)=-x\leqq 0,\; g_2(x,y)=-y\leqq 0$ 
\end{tabular}\hfill {\rm (P$^\pr$)}
%\end{align}
%\begin{flushright} {\rm (P$^\pr$)}\end{flushright}
\medskip

\noindent 
by the substitution $x=x_1^2$, $y=x_2^2$, where 
\[
f=(f_1,f_2),\quad f_1(x,y)=(x+y)^2-2x+2y,\quad f_2(x,y)=(x-1)^2+2y.
\]
The only stationary point of {\rm (P$^\pr$)} is $(x,y)=(1,0)$ and it is weakly efficient. Therefore, {\rm (P$^\pr$)} is KT-pseudoinvex-I in contrast of the unconstrained minimization of $f(x_1,x_2)$.
\end{example}

\begin{example}\label{ex4}
Consider the  problem (P) where $X\equiv\R^2$, $f=(f_1,f_2)$  is the vector function of two variables such that 
\[f_1(x_1,x_2)=2x_1x_2-2x_1^2-x_2^2+8x_1-6x_2,\quad f_2(x_1,x_2)=-x_1+x_2,\] 
and $g$ is the scalar function of two variables 
\[
g(x_1,x_2)=x_1-x_1^2+x_2.
\]
The set of the KT stationary points is the segment whose endpoints are (1,0)  and (1,-2):
\[
\{(x_1,x_2)\mid x_1=1,-2\le x_2\le 0\}.
\]
 All points from the segment are not weakly efficient, because both functions $f_1$ and $f_2$ are strictly monotone over the ray  
\[
\{(x_1,x_2)\mid x_1=1, x_2\le -2\},
\]
they approach $-\infty$ when $x_2$ tends to $-\infty$, and all points $(x_1,x_2)$ such that $x_1=1$, $x_2\leqq 0$ are feasible. Therefore, by Proposition \ref{CharKT-I}, the problem is not KT-pseudoinvex-I. All KT stationary points do not satisfy the second-order necessary optimality conditions in Theorem \ref{NKT}. Therefore,  the problem has no  second-order KT stationary points. By Theorem \ref{th2} the problem is second-order KT-pseudoinvex-I.
\end{example}

The next definition was introduced by Arana-Jimenez, Rufian-Lizana, Osuna-Gomez and Ruiz-Garzon \cite{ara08}:

%\begin{definition}[\cite{ara08}]\label{df5}
\begin{definition}\label{df5}
Let the  problem {\rm (P)} be Fr\'echet differentiable.
Then it is called KT-pseudoinvex-II if  there exists a map $\eta: X\times X\to\R^s$ 
such that the following implication holds:
\[
\left.
\begin{array}{l}
x\in X,\; y\in X,\;  f(y)\le f(x) \\
g(x)\leqq 0,\; g(y)\leqq 0 \\
\end{array}\right]
\quad\Rightarrow \quad
\left[
\begin{array}{l}
\nabla f(x)\eta(x,y)<0 \\
\nabla g_{j}(x)\eta(x,y)\leqq 0,\; j\in A(x).
\end{array}
\right.
\]
\end{definition}

Definition \ref{df5} is an another generalization of KT-invexity (see Martin \cite{mar85}) to vector problems.
We introduce the following notion which is a second-order analog of the last definition.

\begin{definition}
Let the  problem {\rm (P)} be Fr\'echet differentiable.
Then we call {\rm (P)} second-order KT-pseudoinvex-II if  there exist maps $d: X\times X\to\R^s$,
$\eta: X\times X\to\R^s$ and a function $\omega: X\times X\to [0,+\infty)$
such that the following implication holds:
\[
\left.
\begin{array}{l}
x\in X,\; y\in X \\
g(x)\leqq 0,\\
g(y)\leqq 0, \\
f(y)\le f(x)
\end{array}\right]
\;\Rightarrow \;
\left[
\begin{array}{l}
\exists f^{\pr\pr}(x,d(x,y)),\;\exists g^{\pr\pr}_{j}(x,d(x,y)),\; j\in A(x) \\
\nabla f(x)\eta(x,y)+
\omega(x,y) f^{\pr\pr}(x,d(x,y))<0 \\
\nabla g_{j}(x)\eta(x,y)+
\omega(x,y) g^{\pr\pr}_{j}(x,d(x,y))\leqq 0, \; j\in A(x) \\
\textrm{where}\quad d(x,y)\in D(x).
\end{array}
\right.
\]
We suppose here that $f^{\pr\pr}_i(x,d(x,y))$ and $g^{\pr\pr}_j(x,d(x,y))$ are finite.
\end{definition}

The proofs of the following results are similar to the proofs of Theorems \ref{KT-I} and \ref{th2}.

\begin{proposition}
Every Fr\'echet differentiable KT-pseudoinvex-II problem with inequality constraints is second-order KT-pseudoinvex-II.
\end{proposition}

\begin{proposition}
Let {\rm (P)} be a Fr\'echet differentiable problem. Then each second-order  KT stationary point is a global Pareto minimizer
if and only if {\rm (P)} is second-order KT-pseudoinvex-II.
\end{proposition}

\section{The Weighting Problem and Second-Order Kuhn-Tucker Stationary Points}
\label{s5}\setcounter{equation}{0}

For every vector $\lambda\in\Lambda$ consider the weighting scalar problem P$_\lambda$. The following notion was introduced by Osuna-Gomez, Beato-Moreno, Rufian-Lizana \cite{osu99}:

%\begin{definition}[\cite{osu99,ara08}]\label{df7}
\begin{definition}\label{df7}
Let the  problem {\rm (P)} be Fr\'echet differentiable.
Then it is called KT-invex if  there exists a map $\eta: X\times X\to\R^s$ 
such that the following implication holds:
\[
\left.
\begin{array}{l}
x\in X,\; y\in X,\;   \\
g(x)\leqq 0,\; g(y)\leqq 0 \\
\end{array}\right]
\quad\Rightarrow \quad
\left[
\begin{array}{l}
 f(y)-f(x)\geqq\nabla f(x)\eta(x,y) \\
0\geqq\nabla g_{j}(x)\eta(x,y),\; j\in A(x).
\end{array}
\right.
\]
\end{definition}

Definition \ref{df7} is a generalization of the notion of KT-invex problems due to Martin \cite{mar85} for scalar problems.

It is shown by Osuna-Gomez, Beato-Moreno, Rufian-Lizana \cite{osu99} that a Fr\'echet differentiable problem (P) is KT-invex if and only if every KT stationary point is a solution of the weighting problem. 
We introduce the following notion, which is a generalization of the notion of second-order KT-invex problems with inequality constraints for scalar problems (see Ivanov \cite{JOGO,Optimization}).

\begin{definition}\label{df1}
Let the problem {\rm (P)} be Fr\'echet differentiable. Then we call {\rm (P)}
second-order KT-invex if there are
maps $d: X\times X\to\R^s$, $\eta: X\times X\to\R^s$ and a function
$\omega: X\times X\to [0,+\infty)$ such that $f^{\pr\pr}(x,d(x,y))$ and $g^{\pr\pr}_{j}(x,d(x,y))$, $j\in A(x)$
exist as finite values for all $x\in S$, $y\in S$ and the following inequalities hold:
\begin{equation}\label{25}
f_i(y)-f_i(x)\geqq\nabla f_i(x)\eta(x,y)+\omega(x,y) f^{\pr\pr}_i(x,d(x,y)),\; i=1,2,\ldots,n,
\end{equation}
\begin{equation}\label{26}
0\geqq\nabla g_j(x)\eta(x,y)+\omega(x,y) g^{\pr\pr}_j(x,d(x,y)),\; j\in A(x).
\end{equation}
We suppose here that the direction $d(x,y)$ is critical at $x$.
\end{definition}

In this definition the second-order directional derivatives of $f$ and $g$ are not required to exist for every direction.

\begin{theorem}\label{th5}
Every Fr\'echet differentiable KT-invex vector problem with inequality constraints is second-order KT-invex.
\end{theorem}
%\begin{proof}\smartqed
\begin{proof}
Let (P) be KT-invex. Therefore, for all $x\in S$, $y\in S$ there exists
$u\in\R^n$ such that
\[
 f(y)-f(x)\geqq \nabla f(x)u,\quad 0\geqq\nabla g_{j}(x)u,\; j\in A(x).
\]
The choice $\eta=u$, $\omega=0$, and $d=0$ shows that (P) is second-order KT
invex,  because $f^{\pr\pr}(x,0)=0$ and $g_{j}^{\pr\pr}(x,0)=0$, $j\in A(x)$.
\end{proof}

\begin{theorem}\label{th6}
Let {\rm (P)}  be Fr\'echet %and second-order directionally
differentiable. Then each second-order  KT stationary point $x$ is a global solution of the weighting problem ${\rm (P_\lambda)}$ with the same vector $\lambda$
if and only if {\rm (P)} is second-order KT-invex.
\end{theorem}
\begin{proof}
Let (P) be second-order KT-invex.  We prove that every second-order KT stationary point $x$ is a global solution of the respective weighting problem. Choose arbitrary feasible point $y\in S$.  Suppose that  $x$ is a second-order KT point. It follows from KT-invexity that there exists a direction  $d\in\R^s$ which is critical at $x$, a map $\eta\in\R^s$, and a number $\omega\geqq 0$ such that the derivatives $f^{\pr\pr}(x,d)$, $g^{\pr\pr}_{j}(x,d)$, $j\in A(x)$ exist and  equalities (\ref{25}) and (\ref{26}) are satisfied.
We conclude from $x$ is a second-order KT point and $d$ is critical that there exist $\lambda\ge 0$ and $\mu\geqq 0$ satisfying (\ref{23}) and (\ref{24}).
Let us multiply (\ref{25}) by $\lambda_i$, (\ref{26}) by $\mu_j$ and add all obtained inequalities. Taking into account (\ref{23}) and (\ref{24}) we conclude that $\scalpr{\lambda}{f(y)}\geqq\scalpr{\lambda}{f(x)}$.
 
Let each second-order KT stationary point be a solution of the weighting problem. We prove that (P) is second-order KT invex. Suppose the contrary that there exist $x\in S$, $y\in S$ such that for every critical direction $d$ with the property that, if the derivatives $f^{\pr\pr}(x,d)$, $g^{\pr\pr}_{j}(x,d)$, $j\in A(x)$ exist with finite values, then  the system
\[
\left\{
\begin{array}{l}
\nabla f(x)\eta+\omega f^{\pr\pr}(x,d)\leqq f(y)-f(x) \\
\nabla g_{j)}(x)\eta+\omega g^{\pr\pr}_{j)}(x,d)\leqq 0,\; j\in A(x)
\end{array}\right.
\]
has no solution for $\eta\in\R^s$ and $\omega\geqq 0$. Therefore, the linear programming problem 

\medskip
\begin{tabular}{ll}
Maximize & 0         \\
subject to &$\nabla f_i(x)\eta +\omega f^{\pr\pr}_i(x,d)\leqq f_i(y)-f_i(x)$, $i=1,2,\ldots,n$,  \\
& $\nabla g_j(x)\eta +\omega g^{\pr\pr}_j(x,d)\leqq 0,\; j\in A(x)$, \\
& $\omega\geqq 0$.
\end{tabular}
\medskip

\noindent
is infeasible. Its dual linear programming problem is the following one

 \medskip
\begin{tabular}{ll}
Minimize & $\sum_{i=1}^n\lambda_i[f_i(y)-f_i(x)]$ \\
subject to &
$\sum_{i=1}^n\lambda_i\nabla f_i(x)+\sum_{j\in A(x)}\mu_j\nabla g_j(x)=0$, \\
& $\sum_{i=1}^n\lambda_i f^{\pr\pr}_i(x,d)+\sum_{j\in A(x)}\mu_j g^{\pr\pr}_j(x,d)\ge 0,$ \\
& $\lambda_i\geqq 0,\; i=1,2,\dots, n,\quad  \mu_j\geqq 0,\; j\in A(x)$.
\end{tabular}
\medskip

It follows from Duality Theorem that both problems are simultaneously solvable or unsolvable. The dual problem is feasible ($\lambda=0$, $\mu=0$ is a feasible point).  Therefore, the dual problem is not solvable, because it is unbounded. Hence, there exists a feasible point $(\lambda,\mu)$ for the dual problem such that $\lambda\ne 0$. It follows from $d$ is critical and $(\lambda,\mu)$ with $\lambda\ne 0$ is feasible for the dual problem that $x$ is a second-order Kuhn-Tucker point for the problem (P). According to our assumption $x$ is a solution of the weighting problem, that is $\scalpr{\lambda}{f(y)-f(x)}\geqq 0$. This conclusion is satisfied for for every $(\lambda,\mu)$ with $\lambda\ne 0$. Hence the optimal value of the dual problem is bounded from below by 0, which contradicts our indirect conclusion that it is unbounded.
\end{proof}

Example \ref{ex2} show that the converse claim in Theorem \ref{th5} is not satisfied. 
The stationary point $x=(0,0)$ is not weakly efficient. Therefore, it is not a global solution of the scalar problem $P_\lambda$ for every $\lambda\in\Lambda$, and by Theorem \ref{th6}, the unconstrained problem from this example is not KT-invex. The set of second-order stationary points consists of the points $(1,0)$ and $(-1,0)$. Both points are global minimizers of $f_1$ and $f_2$. Therefore, these points are global solutions of $P_\lambda$.                                                                                                                                                  
%Consider Example \ref{ex1}. The stationary point $x=0$ is not weakly efficient. Therefore, it is not a global solution of the scalar problem $P_\lambda$ for every $\lambda\in\Lambda$, and by Theorem \ref{th6}, the unconstrained problem from this example is not KT-invex. Every point $z\in(-\infty,-1]\cup [1,+\infty)$ is a second-order stationary point with $\lambda=(\lambda_1,\lambda_2)$ where $\lambda_1=1$, $\lambda_2=2(z^2-1)$. We can see immediately that $z$ is a global minimizer of $L=\scalpr{\lambda}{f}$.
%Consider Example \ref{ex2}. 

\end{document}